\documentclass{amsart}
\usepackage{enumitem}
\usepackage[all]{xy} 
\usepackage{multirow} 
\usepackage{tabu}
\usepackage{leftindex}
\usepackage{hyperref}

\newtheorem{theorem}{Theorem}[section]
\newtheorem{lemma}[theorem]{Lemma}

\newtheorem{prop}[theorem]{Proposition}

\theoremstyle{definition}
\newtheorem{definition}[theorem]{Definition}
\newtheorem{example}[theorem]{Example}

\numberwithin{equation}{section}

\newcommand{\Z}{\mathbb{Z}}
\newcommand{\coh}{\mathcal{H}^0}
\newcommand{\zcoh}{\widehat{\mathcal{H}}^0}
\newcommand{\fcoh}{\mathcal{H}^1}

\newcommand{\Br}{\operatorname{Br}}
\newcommand{\rBr}{\widehat{\operatorname{Br}}}
\newcommand{\chars}{\mathcal{X}}
\newcommand{\torus}{\mathcal{D}}
\newcommand{\permcat}[1]{#1\textbf{-perm}}
\newcommand{\modcat}[1]{#1\textbf{-mod}}
\newcommand{\permcatStab}[1]{#1\textbf{-\underline{perm}}}
\newcommand{\modcatStab}[1]{#1\textbf{-\underline{mod}}}
\newcommand{\Modcat}[1]{#1\textbf{-Mod}}

\newcommand{\legend}[1]{\ensuremath{\mathsf{#1}}}

\title{Classifying torsors of tori with Brauer groups}

\author[Alexander Duncan]{Alexander Duncan}
\author[Pankaj Singh]{Pankaj Singh}

\address{Department of Mathematics, University of South Carolina, 
Columbia, SC 29208}
\email{duncan@math.sc.edu}
\email{pksingh@email.sc.edu}

\date{\today}

\begin{document}

\begin{abstract}
Using Mackey functors, we provide a general framework for classifying
torsors of algebraic tori in terms of Brauer groups of finite field
extensions of the base field.
This generalizes Blunk's description of the tori associated to del Pezzo
surfaces of degree 6 to all retract rational tori,
essentially the largest class for which this is possible.
\end{abstract}

\maketitle

\section{Introduction}
\label{sec:intro}

Suppose $T$ is an algebraic torus over a field $k$.
Recall that the first Galois cohomology group $H^1(k,T)$
classifies $T$-torsors up to isomorphism.
The purpose of this paper is to investigate isomorphisms of the form 
\begin{equation} \label{eq:main_equation}
H^1(k,T) \cong \operatorname{ker}\left( \prod_{i=1}^r \operatorname{Br}(F_i)
\to \prod_{i=1}^s \operatorname{Br}(L_i)
\right)
\end{equation}
where $F_1,\ldots, F_r, L_1,\ldots, L_s$ are finite separable field
extensions of $k$.

Our inspiration and motivation come from Blunk~\cite{BLUNK}, where
del Pezzo surfaces of degree $6$ over a field $k$ are classified
in terms of certain finite dimensional algebras over $k$.
A (slightly weaker) reformulation of his main result is as follows:

\begin{theorem}[Blunk]
Suppose $T$ is a torus acting faithfully with an open orbit on a del
Pezzo surface of degree $6$ over a field $k$.
Let $F_2/k$ and $F_3/k$ be linearly disjoint field extensions of degree
$2$ and $3$
such that the Galois closure of $F_2F_3/k$ is a splitting field of $T$.
There is an injective group homomorphism
\[
H^1(k,T) \hookrightarrow
\operatorname{Br}(F_2)\times
\operatorname{Br}(F_3)
\]
whose image is the set of pairs $(B,Q)$ such that
\[
\operatorname{Res}_{F_2F_3/F_2}B=0,\
\operatorname{Cor}_{F_2/k}B=0,\
\operatorname{Res}_{F_2F_3/F_3}Q=0,\
\operatorname{Cor}_{F_3/k}Q=0.
\]
\end{theorem}

In \cite{Duncan2016}, the first author showed how
Blunk's idea can be extended to apply to an arbitrary toric variety
when the tori are retract rational.
Following up on this,
in \cite{BDLM}, it was shown that this is essentially the most general
case where the idea could be applied without appealing to
particularities of the base field.

\begin{theorem}[Ballard, Duncan, Lamarche, McFaddin]
Suppose $T$ is a torus over $k$.
The torus $T$ is retract rational if and only if
there are finite separable field extensions $F_1/k, \ldots, F_r/k$
and a group homomorphism  
\[
H^1(k,T) \to \prod_{i=1}^r \Br(F_i)
\]
that is functorial and injective for
all (not necessarily algebraic) field extensions.
\end{theorem}

However, the results for general tori
have two shortcomings when compared with Blunk's.
First, while a map into a product of Brauer groups is known to exist,
little is known about the specific form of its image.
Our first theorem, which is a straightforward application of
standard techniques, remedies this:

\begin{theorem} \label{thm:main_existence}
If $T$ is retract rational,
then an isomorphism \eqref{eq:main_equation} exists.
\end{theorem}

The second shortcoming is that, while the proofs are effective,
the constructions are not especially illuminating.
Indeed, there are infinitely many choices of isomorphism
\eqref{eq:main_equation}.  Even when the groups 
$\{\operatorname{Br}(F_i)\},\{\operatorname{Br}(L_j)\}$ are fixed,
there are many different maps.  Moreover,
it is not easy to see when two different descriptions
of the same map are equal.
Thus, the goal of this paper is to develop computational
tools to understand these maps.
Our main tool is the machinery of Mackey functors.

In order to discuss our results, we briefly overview some facts that are
developed in more detail later in the paper.
A torus $T$ has a Galois splitting field $K/k$ with Galois group $G$.
The character lattice $M$ of $T$ is a $\mathbb{Z}G$-module.
To each $G$-set $X$, we may construct a \emph{permutation lattice}
$\mathbb{Z}[X]$ which is the free abelian group on $X$ with the
evident $G$-action.
There is a canonical bijection between transitive $G$-sets and
field extensions $k \subseteq F \subseteq K$.

A (cohomological) \emph{Mackey functor} can be thought of as an
assignment of an abelian
group to each subgroup $H\le G$ along with restriction, corestriction,
and conjugation maps like in group cohomology.
In particular, for a $G$-module $L$, 
group homology $H_i(-,L)$, group cohomology $H^i(-,L)$,
and Tate cohomology $\widehat{H}^i(-,L)$ are all examples.

\begin{theorem} \label{thm:main_detailed}
Suppose $T$ is retract rational.
Given $G$-sets $X$ and $Y$ and an exact sequence of
Mackey functors
\[
H^0(-,\mathbb{Z}[Y]) \xrightarrow{\psi} H^0(-,\mathbb{Z}[X]) \to H^1(-,M) \to 0,
\]
there is an associated exact sequence
\[
0 \to H^1(k,T) \to \prod_{i=1}^r \Br(F_i)
\xrightarrow{\psi^\ast} \prod_{i=1}^s \Br(L_i)
\]
where $F_1,\ldots, F_r$ (resp. $L_1,\ldots, L_s$) are the
finite separable field extensions corresponding to the orbits of
$X$ (resp. $Y$).
\end{theorem}

The maps $\psi$ and $\psi^\ast$ have several interpretations,
discussed in detail later in the paper.
The data defining the map $\psi$ boils down to a
$\mathbb{Z}G$-module morphism 
$\mathbb{Z}[Y] \to \mathbb{Z}[X]$.
The Brauer groups can be thought of as Galois cohomology groups of
quasitrivial tori and $\psi$ induces a morphism of these tori.
Alternatively, one can describe maps on the Brauer groups directly in
terms of restrictions, corestrictions, and
maps induced by field isomorphisms.
Our theorem gives all maps that occur using any of these interpretations.

The main advantage of Theorem~\ref{thm:main_detailed} is that
it transforms questions about Brauer groups to questions
about integral representation theory of finite groups.
However, by working with \emph{relative} Brauer groups instead,
we can reduce to questions involving modular representation
theory over finite fields, which is even more practical.

For a field $k \subseteq F \subseteq K$, we have the
\emph{relative Brauer group}
\[
\Br(K/F) := \operatorname{ker}\left(
\Br(F) \to \Br(K) \right).
\]
Since $H^1(k,T)=H^1(K/k,T(K))$ in our context,
all the invariants are already captured by the relative Brauer groups.

Our theorems above apply to relative Brauer groups as well.
Moreover, one can use what we call \emph{Tate Mackey functors} in this case,
which provide technical advantages.
In this context, there exists a finite set of ``minimal''
descriptions as in Theorem~\ref{thm:main_detailed},
and they can be computed in practice.

\smallskip

The rest of the paper is organized as follows.
In Section~\ref{sec:coflasque}, we review the theory of flasque and
coflasque lattices.  In Section~\ref{sec:tori}, we review algebraic tori
and their Galois cohomology.
In Section~\ref{sec:existence_proof}, we prove
Theorem~\ref{thm:main_existence} and its relative version.
In Section~\ref{sec:Mackey}, we recall the theory of cohomological
Mackey functors and develop the theory.
In Section~\ref{sec:main_theorem}, we prove
Theorem~\ref{thm:main_detailed} and its relative version.
In Section~\ref{sec:perm_resolutions}, we describe how one can
determine all possible descriptions of $H^1(k,T)$ using modular
representation theory.
Finally, in Section~\ref{sec:blunk_example} we revisit Blunk's theorem
and show how his description (and others) can be (re)discovered
using our theory.

\section{Flasque and coflasque resolutions}
\label{sec:coflasque}

We recall standard facts about the theory of flasque and coflasque
resolutions of $G$-lattices
(see, e.g., \cite{colliot-thelene1977r}
\cite{Colliot-Thelene.Sansuc87}
\cite{Lee09Flasque},
\cite{BDLM}).
Throughout this section, $G$ is a finite group.

Let $R$ be a ring.
Let $RG$ denote the group $R$-algebra of $G$.
An $RG$-lattice is a left $RG$-module
that is finitely generated and free as an $R$-module.
Given a $G$-set $X$, let $R[X]$ denote the free $R$-module with basis
indexed by $X$ with the evident left $G$-action.
A \emph{permutation $RG$-module} is a left $RG$-module of the form
$R[X]$ for some $G$-set $X$.
A \emph{permutation $RG$-lattice} is a finitely generated permutation
$RG$-module, which is, in particular, an $RG$-lattice.

Let $\Modcat{RG}$ denote the category of $RG$-modules, and
$\modcat{RG}$ denote the category of finitely generated
$RG$-modules.
Let $\permcat{RG}$ the full subcategory of permutation $RG$-lattices
in $\modcat{RG}$.
We simply say ``$G$-lattice'' etc.~when the ring $R$ is equal to the
integers $\mathbb{Z}$,
which will be the case for the remainder of this section.

Two $G$-lattices $L_1,L_2$ are \emph{stably permutation equivalent}
if there exist permutation $G$-lattices $P_1,P_2$ such that
$L_1 \oplus P_1 \cong L_2 \oplus P_2$.
By Shapiro's Lemma, for a permutation lattice $P$,
we have $H^1(G,P)=\widehat{H}^{\pm 1}(G,P)=0$.

\begin{definition}
Let $L$ be a $G$-lattice. 
\begin{enumerate}
\item $L$ is \emph{stably permutation} if $L$ is stably permutation
equivalent to a permutation lattice.
\item $L$ is \emph{invertible} if it is a direct summand of a
permutation $G$-lattice.
\item $L$ is {\it flasque} if $\widehat{H}^{-1}(H, L)=0$ for all subgroups $H \le G$.
\item $L$ is {\it coflasque} if $\widehat{H}^1(H, L)=0$ for all subgroups $H \le G$.
\end{enumerate}
\end{definition}

Given a $G$-lattice $L$, the \emph{dual $G$-lattice}
$L^{\vee}:= \operatorname{Hom}_{\Z}(L, \Z)$ is the group of all
group homomorphisms from $L$ to $\Z$ with the induced $G$-action, where
$G$ acts trivially on $\Z$.
Duality induces an exact anti-equivalence of the category of
$G$-lattices with itself.
Moreover, coflasque lattices are dual to flasque lattices.

\begin{definition} \label{def:fl_cofl_res}
Let $L$ be a $G$-lattice.
\begin{enumerate}
\item A \emph{flasque resolution of $L$ of the first type} is a short exact sequence 
\[
0\rightarrow L \rightarrow P \rightarrow F \rightarrow 0,
\]
\item a \emph{flasque resolution of $L$ of the second type} is a short exact sequence 
 \[
0\rightarrow P \rightarrow F \rightarrow L \rightarrow 0,
\]
\item a \emph{coflasque resolution of $L$ of the first type} is a short exact sequence 
\[
0\rightarrow C \rightarrow P \rightarrow L \rightarrow 0,
\]
\item a \emph{coflasque resolution of $L$ of the second type} is a short exact sequence 
\[
0\rightarrow L \rightarrow C \rightarrow P \rightarrow 0,
\]
\end{enumerate}
where, in each of the above sequences, $P$ is a permutation lattice,
$C$ is a coflasque lattice, and $F$ is a flasque lattice.
\end{definition}

\begin{lemma}
Flasque and coflasque resolutions of both types always exist.
\end{lemma}

\begin{proof}
A proof can be found in \cite[0.6]{Colliot-Thelene.Sansuc87},
but we sketch the argument here since the ideas are so closely related
to the material on Mackey functors below.

First, we observe that it suffices to prove the existence for coflasque
resolutions and then the flasque resolutions will follow by duality.

For a coflasque resolution of the first type, we simply need to construct a
permutation lattice $P$ and a morphism $P \to L$ such that
$H^0(H,P) \to H^0(H,L)$ is surjective for all subgroups $H \le G$.
This can be constructed by observing that for each subgroup $H$,
a surjection
$\mathbb{Z}^{n_H} \to L^H$ of abelian groups corresponds to a morphism
$\mathbb{Z}[G/H]^{n_H} \to L$.
Our desired surjection is therefore
\[
\bigoplus_{H \le G} \mathbb{Z}[G/H]^{n_H} \to L .
\]

For coflasque resolutions of the second type,
we have the following commutative diagram with exact rows and columns: 
\begin{equation} \label{eq:cofl2_construction}
\xymatrix{
 & & 0 \ar[d] & 0 \ar[d] & \\
 &  & B \ar[d] \ar@{=}[r] & B \ar[d] & \\
0 \ar[r] & L \ar@{=}[d] \ar[r] & C \ar[d] \ar[r] & Q \ar[d] \ar[r] & 0 \\
0 \ar[r] & L \ar[r] & P \ar[d] \ar[r] & F \ar[d] \ar[r] & 0 \\
 & & 0 & 0 &
}
\end{equation}
where $F$ is flasque, $B$ and $C$ are coflasque, and $P$ and $Q$ are a
permutation lattices.
This is constructed by first taking a flasque resolution of the first
type of $L$ to obtain $P$ and $F$, then taking a coflasque resolution of the
first type of $F$ to obtain $Q$ and $B$, then finally constructing $C$ by
pullback.
Since $B$ is coflasque, $\operatorname{Ext}^1_{\mathbb{Z}G}(P,B)=0$
and we conclude that $C \cong P \oplus B$ is coflasque as well.
\end{proof}

The flasque and coflasque resolutions are never unique, but their
stable permutation equivalence classes are unique
(see \cite[0.6]{Colliot-Thelene.Sansuc87}).
However, the stable permutation equivalence classes of a
flasque/coflasque lattice from different types of resolutions are
\emph{not} necessarily stably permutation equivalent.
Nevertheless, we do have the following:

\begin{lemma} \label{lem:fl_cofl_compare}
Let $L$ be a $G$-lattice.
If the flasque/coflasque lattice from one of the resolutions in
Definition~\ref{def:fl_cofl_res} is invertible,
then the flasque/coflasque lattice in every other resolution is invertible. 
\end{lemma}

\begin{proof}
A lattice is permutation (resp. invertible) if and only its dual lattice
is permutation (resp. invertible).
The result now follows by duality and
the diagram \eqref{eq:cofl2_construction}.
\end{proof}

\section{Algebraic tori}
\label{sec:tori}

Throughout the following $k$ is a field, $k^s$ is its separable closure
and $\Gamma_k$ is its absolute Galois group, which is a profinite group.
Let $K$ be a finite Galois field
extension of $k$ with Galois group $G$.
We assume throughout that all algebras are associative and unital.

\subsection{\'Etale algebras}
\label{ssec:etale}

See \cite[\S{18}]{knus1998book} for more on this material.
An \emph{\'etale $k$-algebra} $E$ is a direct product 
\[
E = F_1 \times \cdots \times F_m
\]
where $F_1, \ldots, F_m$ are finite separable field extensions of $k$.
An \'etale $k$-algebra $E$ is \emph{split} if $E \cong k^n$
for some non-negative integer $n$.
An \'etale $k$-algebra is \emph{split by $K/k$} if $E\otimes_k K$ is
split.

There is a contravariant equivalence between the category of finite
discrete $\Gamma_k$-sets and the category of \'etale $k$-algebras.
Similarly, there is a contravariant equivalence between the category of
finite $G$-sets and the category of \'etale $k$-algebras split by $K$.
In both cases, transitive orbits correspond to fields.

\subsection{Tori}
\label{ssec:tori}

We recall standard results from, e.g.,
\cite{MR1634406} or \cite{Duncan2016}.

Let $\mathbb{G}_m$ be the multiplicative group scheme of the field $k$.
A \emph{(algebraic) torus} over $k$ is an algebraic group $T$ such that
$T_{k^s} := T \times_{\operatorname{Spec}(k)}
\operatorname{Spec}(k^s) \cong \mathbb{G}_{m,k^s}^n$ for some
non-negative integer $n$.
A torus is \emph{split} if $T \cong \mathbb{G}_m^n$ already over the
base field $k$.

A torus is \emph{split by $K$} if $T_K \cong \mathbb{G}_{m,K}^n$ for some
non-negative integer $n$ and some field extension $K$.
For every torus $T$, there is a unique minimal field extension $K/k$
such that $T$ is split by $K$.  We call this minimal field extension the
\emph{splitting field of $T$}.  The splitting field is always finite
Galois.

Every torus $T$ has an associated \emph{character lattice} $\chars(T)$,
which is a discrete $\Gamma_k$-lattice.
Conversely, every discrete $\Gamma_k$-lattice $L$ has an associated torus
$\torus(L)$.
There is an exact anti-equivalence between the category of discrete
$\Gamma_k$-lattices and the category of tori over $k$.
This restricts to an exact anti-equivalence between the category of
$G$-lattices and the category of $k$-tori split by $K$.
Since every torus has a finite Galois splitting field,
the action of $\Gamma_k$ on its character lattice
always factors through a finite group.

\begin{definition}
Let $T$ be a torus with character lattice $M:= \chars(T)$.
\begin{enumerate}
\item $T$ is \emph{quasi-trivial} if $M$ is permutation.
\item $T$ is \emph{special} if $M$ is invertible.  
\item $T$ is \emph{flasque} if $M$ is flasque. 
\item $T$ is \emph{coflasque } if $M$ is coflasque. 
\end{enumerate} 
\end{definition}

We have the following standard fact
(see \cite{huruguen2016special} and \cite{Merku}):

\begin{prop} \label{prop:special}
A torus $T$ is special if and only if $H^1(F,T_F)$ is trivial for every
(not necessarily algebraic) field extension $F/k$.
\end{prop}

Given a $G$-set $X$, we obtain a permutation $G$-lattice
$\mathbb{Z}[G]$.
Thus, given an \'etale $k$-algebra $E$, we obtain a
quasi-trivial torus, namely the \emph{Weil restriction}
$\mathbb{R}_{E/k} \mathbb{G}_m$.

\subsection{Brauer groups and Galois cohomology}

Let $T$ be a $k$-torus with splitting field $K/k$
and Galois group $G$.
We have Galois cohomology groups
\[
H^i(k,T) = H^i(\Gamma_k,T(k_s))
\textrm{ and }
H^i(K/k,T) = H^i(G,T(K)) .
\]
In particular, we have cohomological descriptions of
the Brauer group
\[
\Br(k) \cong H^2(k,\mathbb{G}_m) = H^2(\Gamma_k,k_s^\times),
\]
and the relative Brauer group
\[
\Br(K/k) \cong H^2(K/k,\mathbb{G}_m) = H^2(G,K^\times).
\]
We use the notation $\rBr(k) := \Br(K/k)$ for the relative Brauer group
whenever the overfield $K$ is clear.

Recall some standard consequences of Hilbert's Theorem 90
and Shapiro's Lemma:

\begin{prop}[\cite{BDLM}]
Let $X$ be a $G$-set of cardinality $n$ with orbits $X_1, \ldots, X_r$.
Let $E= F_1 \times \cdots \times F_r$ be the corresponding
\'etale $k$-algebra of degree $n$,
with field extensions $F_1, \ldots, F_r$.
Let $T=\mathbb{R}_{E/k} \mathbb{G}_m$ be the
corresponding quasitrivial torus.
Then 
\begin{enumerate}
    \item $\chars(T) \cong \mathbb{Z}[X]$,
    \item $T(k) \cong E^{\times}$,
    \item $H^1(k,T)=0$, and
    \item $H^2(k,T)=\prod_{i=1}^{r}\Br(F_i)$.
\end{enumerate}
\end{prop}

In view of the preceding, we will use the following shorthand notations.
Given a $G$-set $X$ with \'etale $k$-algebra $E$, we write
\begin{enumerate}
\item $\torus(X) := \torus(\mathbb{Z}[X])$,
\item $\Br(E) = \Br(X) := H^2(k,\torus(X))$, and
\item $\rBr(E) = \rBr(X) := H^2(K/k,\torus(X))$.
\end{enumerate}

\subsection{Retract rationality}
\label{ssec:rrat}

We recall some standard terminology regarding variants of rationality
studied in algebraic geometry:
\begin{definition}
Let $X$ be a $k$-variety.
\begin{itemize}
\item $X$ is \emph{rational} if $X$ is birationally
equivalent to $\mathbb{A}_k^n$ for some $n\geq 0$.
\item $X$ is \emph{stably rational } if
$X \times \mathbb{A}_k^m $ is rational for some $m \ge 0$.
\item $X$ is \emph{retract rational} if there exists an integer $n$,
a dominant rational map
$f:\mathbb{A}_k^n \dasharrow X$ and a rational map $s: X
\dasharrow \mathbb{A}_k^n $ such that $f\circ s$ is the identity on
$X$.
\end{itemize}
\end{definition}

\begin{theorem}
\label{thm:retrat_inv}
Let $T$ be a $k$-torus with character lattice $M:=\chars(T)$,
and suppose 
\[
1\rightarrow M\rightarrow P \rightarrow F \rightarrow 1
\]
is a flasque resolution of the first type. 
\begin{enumerate}
\item $T$ is stably rational if and only if $F$ is stably permutation. 
\item $T$ is retract rational if and only if $F$ is invertible.
\end{enumerate}
\end{theorem}

\begin{proof}
The stable rationality statement is essentially due to
Voskresenski\u{i}~\cite[Theorem 2]{Voskresenskii_1974}.  
The retract rationality statement is due to
Saltman~\cite[Theorem 3.14]{Saltman1984}.
\end{proof}

\section{Proof of the Existence Theorem}
\label{sec:existence_proof}

Here we prove Theorem~\ref{thm:main_existence},
which is an immediate consequence of the following stronger result.

\begin{theorem} \label{thm:main_existence_relative}
Suppose $T$ is retract rational.
There exist $G$-sets $X_0$ and $X_1$ and a $G$-module homomorphism
$\psi : \mathbb{Z}[X_1] \to \mathbb{Z}[X_0]$ with isomorphisms
\begin{align}
H^1(k, T) \cong&
\operatorname{ker}\left(\psi^\ast : \Br(X_0)\to \Br(X_1)\right)
\label{eq:abs_map}
\\
\cong&
\operatorname{ker}\left(\psi^\ast : \rBr(X_0)\to \rBr(X_1)\right)
\label{eq:rel_map} .
\end{align}
\end{theorem}

\begin{proof}
We have a coflasque resolution of $M$ of the second type 
\begin{equation}\label{eq2}
0 \rightarrow M \rightarrow C \rightarrow P \rightarrow 0 
\end{equation}
where $C$ is invertible by Lemma~\ref{lem:fl_cofl_compare}
and Theorem~\ref{thm:retrat_inv}.
Applying duality and Galois cohomology,
we obtain the following exact sequence
\begin{equation}\label{eq3}
H^1(k, \torus(C)) \to H^1(k, \torus(M)) \to H^2(k, \torus(P)) \to
H^2(k, \torus(C))
\end{equation}
Since $C$ is invertible, $D(C)$ is special and this implies that  $H^1(k, \torus(C))=0$.

In addition, there exists a permutation lattice $Q$ such that $C$ is a
direct summand of $Q$.
Define the map $\psi: P \to Q$ as the composition of
$P \to C$ and $C \to Q$.
Moreover, the factorization of the identity
map on $C$ through $Q$ gives a factorization of the identity map 
\begin{equation}\label{eq4}
H^2(k, \torus(C)) \rightarrow H^2(k, \torus(Q)) \rightarrow
H^2(k, \torus(C)). 
\end{equation}
By combining the above equations \eqref{eq3} and \eqref{eq4}, we obtain
the following exact sequence 
\[
0 \rightarrow H^1(k, \torus(M)) \rightarrow H^2(k, \torus(P))
\rightarrow H^2(k, \torus(Q)).
\]
After making canonical identifications, this establishes the isomorphism
\eqref{eq:abs_map}.

We now show that the second isomorphism \eqref{eq:rel_map}
follows from the first \eqref{eq:abs_map}.

If $Y$ is a transitive $G$-set corresponding to a field extension
$F/k$, then we have a surjective map $\pi : \mathbb{Z}[G] \to \mathbb{Z}[Y]$
that corresponds to the restriction map in the exact sequence
\[
0 \to \rBr(F) \to \Br(F) \xrightarrow{\pi^\ast} \Br(K).
\]
Therefore,
we have the following commutative diagram with exact rows
\[
\xymatrix{
0 \ar[r] &
\rBr(X_0) \ar[r] \ar[d]^{\Psi^\ast} &
\Br(X_0) \ar[r] \ar[d]^{\Psi^\ast} &
\Br(K)^r \\
0 \ar[r] &
\rBr(X_1) \ar[r]  &
\Br(X_1) \ar[r] &
\Br(K)^s} .
\]
Since the splitting field of $T$ is $K$, the isomorphism
\eqref{eq:rel_map} follows from \eqref{eq:abs_map}.
\end{proof}

In the proof above, we could have instead established
the relative Brauer iomorphism \eqref{eq:rel_map} first,
and then used that to obtain the version for
absolute Brauer groups
\[
0 \to H^1(k, T) \to \Br(X_0) \to
\Br(X_1) \times \Br(K)^r
\]
by ensuring we split all the classes in $\operatorname{Br}(X_0)$
by $K$.
However, this process highlights that the relationship between
descriptions using absolute Brauer groups and descriptions using
relative Brauer groups is not completely trivial.
See Section~\ref{ssec:blunk_compare} for a concrete example of this phenomenon.

\section{Mackey functors}
\label{sec:Mackey}

Here we introduce the notion of a Mackey functor.
Our main reference is \cite{T-W}.
Throughout this subsection, $R$ is a commutative ring and $G$ is a finite group.

\begin{definition}
  A {\it Mackey functor} for a group $G$ over a commutative ring $R$ is a function 
\begin{align*}
    \mathcal{M}:\left\{ \textup{subgroups of $G$} \right\} \rightarrow \modcat{R}
\end{align*}
with morphisms
\setlength{\abovedisplayskip}{0pt}
\begin{align*}
I_{H}^K: \mathcal{M}(H) &\rightarrow \mathcal{M}(K)\\
R_{H}^K: \mathcal{M}(K) &\rightarrow \mathcal{M}(H)\\
C_g : \mathcal{M}(H)& \rightarrow \mathcal{M}(\leftindex^g{H})
\end{align*}
 for all subgroups $H\le K \le G$ and $g \in G$, such that
 
 \begin{itemize}
     \item $I^H_H$, $R^H_H$, $C_h: \mathcal{M}(H) \rightarrow \mathcal{M}(H)$ are identity morphisms for all subgroups $H\le G$ and $h \in H$
     
     \item $R^H_J R^K_H = R^K_J$ and $ I^K_H I^H_J = I^K_J$ for all subgroups $J \le H \le K$
     
    \item $C_g C_h = C_{gh}$ for $g, h \in G$
    
    \item $R^{\leftindex^g{K}}_{\leftindex^g{H}} C_g = C_g R^K_H$ and $I^{\leftindex^g{K}}_{\leftindex^g{H}} C_g = C_g I^K_H$ for $g\in G$ and for all subgroups $H \le K$.
    
    \item $\displaystyle{R^K_JI_H^K = \sum_{x\in [J\backslash K/ H]} I^J_{H\cap \leftindex^x\!{J}}\, C_x \,R^H_{H\cap \leftindex^x\!{J}}}$ for all subgroups $J, H \le K$.
 \end{itemize}
\end{definition}

To motivate this definition,
consider the example of the representation ring of $G$
along with the representation rings all its subgroups.
Here $I^\bullet_\bullet$ corresponds to induction,
$R^\bullet_\bullet$ corresponds to restriction, and
$C_\bullet$ corresponds to conjugation.
The last axiom above then models the Mackey Decomposition Theorem,
which explains the name ``Mackey functor.''

\begin{example}
For a $G$-module $M$, the
\emph{fixed point functor} $\operatorname{FP}_M$ is defined as follows.
For each subgroup $H \le G$, we define
$\operatorname{FP}_M(H) = M^H$ as the module of $H$-invariants.
Now the map $C_g : M^H \to M^{\leftindex^g{H}}$ is given by $C_g(m)=gm$,
the map
$R^K_H : M^K \to M^H$ is given by inclusion $R^K_H(m)=m$,
and the map $I^K_H : M^H \to M^K$ corresponds to the \emph{relative transfer}
\[
I^K_H(m) = \operatorname{tr}_H^K(m) := \sum_{[g] \in K/H} gm
\]
where $g$ runs over a choice of representatives for the left cosets
of $H$ in $K$.
\end{example}

\begin{definition}
 A Mackey functor $\mathcal{M}$ for $G$ is called \emph{cohomological} if whenever $H\le K \le G$, the map 
 \[
 I_H^KR^K_H:\mathcal{M}(K)\rightarrow \mathcal{M}(K)
 \] 
 is multiplication by the index $[K:H]$.
\end{definition}

\begin{example}
The prototypical example of a cohomological Mackey functor is
group cohomology.
Indeed, for a $G$-module $M$ and a non-negative integer $i$,
we write $\mathcal{H}^i(M)$ to denote the
cohomological Mackey functor given by 
\[
H \mapsto H^i(H, M)
\]
where $R^K_H$ is restriction, $C_g$ is conjugation, and
$I^K_H$ is the \emph{corestriction} or \emph{transfer} operation.
Similarly, one defines cohomological Mackey functors $\mathcal{H}_i(M)$ for group
homology and $\widehat{\mathcal{H}}^i(M)$ for Tate cohomology.
\end{example}

The zeroth cohomology functor $\mathcal{H}^0(M)$ is just another name
for the fixed point functor $\operatorname{FP}_M$.
There is also a \emph{fixed quotient functor} $\operatorname{FQ}_M$,
which is equal to $\mathcal{H}_0(M)$.
We will use the notation $\mathcal{H}^0(M)$
and $\mathcal{H}_0(M)$ in the remainder of the paper.

\begin{example}
Let $K/k$ be a finite Galois field extensions with Galois group $G$.
We define the \emph{relative Brauer group Mackey functor}
as $\widehat{\mathcal{B}} := \mathcal{H}^2(K^\times)$
where $K^\times$ has the natural $\mathbb{Z}G$-module structure.
Equivalently, one may think of $\widehat{\mathcal{B}}$ as the Mackey
functor given by
\[
H \mapsto \widehat{\operatorname{Br}}(K^H) = \operatorname{Br}(K/K^H).
\]
Here, the map $I^\bullet_\bullet$ corresponds to restriction of fields
$\operatorname{Res}_{\bullet/\bullet}$,
the map $R^\bullet_\bullet$ corresponds
to \emph{corestriction}
$\operatorname{Cor}_{\bullet/\bullet}$, and $C_\bullet$ corresponds to
pullback by $k$-algebra isomorphisms.
\end{example}

\begin{example}
Let $K/k$ be a finite Galois field extensions with Galois group $G$.
We have a canonical surjection $\pi : \Gamma_k \to G$ from the absolute
Galois group of $k$ with kernel equal to the absolute Galois group of $K$.
We define the \emph{Brauer group Mackey functor} $\mathcal{B}$
as follows.
For a subgroup $H$ we define $\mathbb{Z}$-modules
\[
\mathcal{B}(H) := H^2(\pi^{-1}(H),(k^s)^\times) \cong
\operatorname{Br}(K^H) .
\]
The operations $I^J_H$, $R^J_H$, and $C_g$ are defined by lifting the
corresponding operations to $\Gamma_k$ or by interpreting the operations
directly in terms of field extensions as in the relative case.
\end{example}

There is a notion of a \emph{morphism of Mackey functors}
$f : M \to N$ given by a collection of morphisms $M(H) \to N(H)$ where
all the evident diagrams involving $I^\bullet_\bullet$,
$R^\bullet_\bullet$, and $C_\bullet$ commute.
Let $\operatorname{CoMack}_R(G)$ denote the category of cohomological
Mackey functors for $G$ over $R$.

We may consider $\mathcal{H}^0$ as a functor from
$\Modcat{RG}$ to $\operatorname{CoMack}_R(G)$
via
\[
V \mapsto \mathcal{H}^0(V)
\]
and similarly for $\mathcal{H}_0$.
There is also an \emph{evaluation functor}
\[
\operatorname{ev} : \operatorname{CoMack}_R(G) \to \Modcat{RG}
\]
which sends a Mackey functor $\mathcal{M}$ to the $RG$-module $\mathcal{M}(1)$.

\begin{prop}
The functors $\mathcal{H}_0 \vdash \operatorname{ev} \vdash \mathcal{H}^0$
form an adjoint triple.
In other words, $\mathcal{H}_0$ is a left adjoint of evaluation
and $\mathcal{H}^0$ is a right adjoint of evaluation.
\end{prop}

\begin{proof}
This is proved in \cite[6.1]{Thvenaz1990SimpleMF} for the category of
\emph{all} Mackey
functors, but $\mathcal{H}_0$ and $\mathcal{H}^0$ factor through
the full subcategory of cohomological Mackey functors.
\end{proof}

In particular, since $\operatorname{ev} \circ \mathcal{H}^0$ is naturally isomorphic
to the identity functor, we see that $\mathcal{H}^0$ is fully faithful.

\subsection{Yoshida's Theorem}
\label{ssec:yoshida}

For a finite group $G$ and a commutative ring $R$, let
$\gamma_R(G)$ be the $R$-algebra
\[
\gamma_R(G) := \operatorname{End}_{RG} \left(
\bigoplus_{H\le G} \operatorname{Ind}_{H}^G(R)
\right)^{\operatorname{op}}.
\]
As an $R$-module, the ring $\gamma_R(G)$ has a basis indexed by elements
$[HgK]$ where $H$ and $K$ vary over all subgroups of $G$ and $g$ varies
over a choice of representatives from each double coset in
$H\backslash G/K$.

We recall the following theorem of Yoshida~\cite{YoshidaII}:

\begin{theorem}[Yoshida] \label{thm:ordinary_Yoshida}
There are equivalences of categories between
\begin{enumerate}
\item the category $\operatorname{CoMack}_R(G)$ of
cohomological Mackey functors,
\item the category $\Modcat{\gamma_R(G)}$ of left $\gamma_R(G)$-modules,
and
\item the category
\[
\operatorname{Add}_R\left(\permcat{RG},\Modcat{R}\right)
\]
of $R$-linear functors from the category of finitely generated
permutation $RG$-modules to the category of $R$-modules.
\end{enumerate}
\end{theorem}

The equivalence $\operatorname{CoMack}_R(G) \to \Modcat{\gamma_R(G)}$
is given by taking
\[
\mathcal{M} \mapsto \bigoplus_{H\le G} \mathcal{M}(H).
\]
The equivalence $\operatorname{CoMack}_R(G) \to 
\operatorname{Add}_R\left(\permcat{RG},\Modcat{R}\right)$
is given by
\[
\mathcal{M} \mapsto
\left(R[X] \mapsto
\operatorname{Hom}_{\operatorname{CoMack}_R(G)}
(\coh(R[X]),\mathcal{M})
\right)
\]
as described in \cite[3.4]{Bouc}.

These different categories make it easier to see certain facts about
cohomological Mackey functors.
For example, by trading
$I^\bullet_\bullet$ and $R^\bullet_\bullet$,
taking $C_g$ to $C_{g^{-1}}$, and reversing the order of composition, we
have a notion of ``contravariant Mackey functor.''
These correspond to right $\gamma_R(G)$-modules
and contravariant $R$-linear functors.
The fact that permutation lattices are isomorphic to their own dual
representations emphasizes the symmetry in a different way.

The category $\operatorname{CoMack}_R(G)$ is an abelian category with enough
projectives since it is equivalent to a category of modules.
Indeed, we have the following description of the
projective objects.

\begin{theorem} \cite[16.5]{T-W} \label{thm:comack_projective}
A cohomological Mackey functor is projective in $\operatorname{CoMack}_R(G)$
if and only if it is isomorphic to $\mathcal{H}^0(P)$ where
$P$ is a direct summand of a permutation
$RG$-module.
\end{theorem}

There is an explicit description of the indecomposable projective
cohomological Mackey functors, which we shall see
in Section~\ref{ssec:tsm}.
    
\subsection{Tate Mackey functors}
\label{ssec:tate}

Here we introduce \emph{Tate Mackey functors}.
While we seem to be the first to use this name, we are certainly not the
first to study them (see, e.g., \cite{Linckelmann16}). 

\begin{definition}
A \emph{Tate Mackey functor $\mathcal{M}$ for $G$} is a cohomological Mackey
functor such that $\mathcal{M}(1)=0$.
\end{definition}

\begin{example}
The functors $\widehat{\mathcal{H}}^i(M)$
are Tate Mackey functors for all $i$.
Therefore, $\mathcal{H}^i(M)$ is a Tate Mackey functor for $i > 0$.
As a particular case, the relative Brauer functor $\widehat{\mathcal{B}}$
is a Tate Mackey functor.
\end{example}

\begin{definition}
We define $\tau_R(G) :=
\gamma_R(G)/\langle [1e1] \rangle$ as the quotient
of Yoshida's algebra $\gamma_R(G)$ by the two-sided ideal generated by
$[1e1]$.
\end{definition}

Let $\operatorname{TMack}_R(G)$ denote the full subcategory of
Tate Mackey functors for $G$ over $R$.
Let $\modcatStab{RG}$ (resp. $\permcatStab{RG}$)
denote the \emph{stable module category} of
$\modcat{RG}$ (resp. $\permcat{RG}$) (see, e.g. \cite[IV.1]{ARS}).
We have the following ``stable'' version of Yoshida's theorem:

\begin{theorem} \label{thm:stable_Yoshida}
There are equivalences of categories between
\begin{enumerate}
\item the category $\operatorname{TMack}_R(G)$ of
Tate Mackey functors,
\item the category $\Modcat{\tau_R(G)}$ of left $\tau_R(G)$-modules,
and
\item the category
\[
\operatorname{Add}_R\left(\permcatStab{RG},\Modcat{R}\right)
\]
of $R$-linear functors from the stable category of permutation
$RG$-lattices to the category of $R$-modules.
\end{enumerate}
\end{theorem}

\begin{proof}
Each of these categories is a full subcategory of the corresponding
category in Yoshida's Theorem. 
Thus, we simply need to verify that the equivalences of Yoshida's
theorem preserve these subcategories.

Multiplication by $[1e1]$ in $\gamma_R(G)$ corresponds to application of
the morphism $I^1_1 : \mathcal{M}(1) \to \mathcal{M}(1)$
in a cohomological Mackey functor.
From this, it follows that the first two categories are equivalent.
Under Yoshida's correspondence, the condition that $F(RG) = 0$ for the
additive category is equivalent to asking that $F(1) = 0$ for the
cohomological functor since $R[G] = R[G/1]$.
\end{proof}

Applying this to the usual extension/restriction adjunction
for the defining surjective $R$-algebra homomorphism
$\gamma_R(G) \to \tau_R(G)$ implies the following:

\begin{prop} \label{prop:adjoint_CT}
There is an adjoint pair $\pi \dashv \iota$ of functors
\[
\pi : \operatorname{CoMack}_R(G) \to\operatorname{TMack}_R(G),
\quad
\iota : \operatorname{TMack}_R(G) \to\operatorname{CoMack}_R(G)
\]
where
$\pi$ is the quotient functor $M \mapsto M/ \langle M(1) \rangle$ and
$\iota$ is the inclusion functor.
\end{prop}

In particular, the functor
$\widehat{\mathcal{H}}^0 : \modcat{R} \to \operatorname{TMack}_R(G)$
is full.

\section{Proof of the Description Theorem}
\label{sec:main_theorem}

Here we proof Theorem~\ref{thm:main_detailed}.
In fact, we simultaneously prove a relative version.

For the rest of the paper, we use the shorthand notation
\[
\mathcal{H}^0(X) := \mathcal{H}^0(\mathbb{Z}[X])
\textrm{ and }
\widehat{\mathcal{H}}^0(X) := \widehat{\mathcal{H}}^0(\mathbb{Z}[X])
\]
whenever $X$ is a $G$-set.

\begin{theorem} \label{thm:main}
Let $G = \textup{Gal}(K/k)$ and $T$ be a retract rational torus over $k$
with character $G$-lattice $M$.
Suppose $X_0$ and $X_1$ are $G$-sets.
If there exists an exact sequence of cohomological Mackey functors 
\begin{equation} \label{eq:start_res}
\mathcal{H}^0(X_1)  \xrightarrow{\psi}
\mathcal{H}^0(X_0)\xrightarrow{\epsilon}\mathcal{H}^1(M)\rightarrow 0
\end{equation}
then there is an exact sequence 
\begin{equation} \label{eq:end_res}
0\rightarrow H^1(k, T) \rightarrow \Br(X_0)\xrightarrow{\Psi} \Br(X_1).
\end{equation}
Similarly, if there exists an exact sequence of Tate Mackey functors 
\begin{equation} \label{eq:start_res_rel}
\zcoh(X_1)  \xrightarrow{\psi} \zcoh(X_0)\xrightarrow{\epsilon} \fcoh(M)\rightarrow 0
\end{equation}
then there is an exact sequence 
\begin{equation} \label{eq:end_res_rel}
0\rightarrow H^1(k, T) \rightarrow \widehat{\textup{Br}}(X_0)\xrightarrow{\Psi} \widehat{\textup{Br}}(X_1).
\end{equation}
\end{theorem}

\begin{proof}
We will prove only the relative Brauer
statement using \eqref{eq:start_res_rel}
since the two proofs are similar.

First, we show that the $G$-module homomorphism from
Theorem~\ref{thm:main_existence_relative} already gives
rise to a sequence as in \eqref{eq:start_res_rel}.
Indeed, the theorem says that there are $G$-sets $Y_0$ and $Y_1$
with a $G$-module homomorphism
$\Phi: \Z[Y_1] \rightarrow \Z[Y_0]$ such that the following is exact
\begin{equation} \label{eq:reminder}
0\rightarrow H^1(k, T) \rightarrow \rBr(Y_0)\xrightarrow{\Psi}
\rBr(Y_1).
\end{equation}
The sequence comes from a coflasque resolution of the second type so we
have an exact sequence of $G$-lattices
\[
0 \to M \to C \to \mathbb{Z}[Y_0] \to 0
\]
with $C$ coflasque and a surjection $\mathbb{Z}[Y_1] \to C$.
Thus, we have an exact sequence of Mackey functors
\[
\zcoh(Y_1) \xrightarrow{\Phi_\ast} \zcoh(Y_0)\xrightarrow{\epsilon} \fcoh(M) \rightarrow 0
\]
such that \eqref{eq:reminder} is exact.
We now use the established sequence using $\{Y_i\}$ to compare to the
given sequence using $\{X_i\}$.

We have two projective presentations of $\fcoh(M)$ in the category of
Tate Mackey functors, which extend to projective resolutions
for appropriate $X_2,Y_2,\ldots$.
Therefore, we have two chain complexes of
projective objects
\[
\xymatrix{
\cdots \ar[r] &
\zcoh(Y_2) \ar[r] \ar[d]^{f_2} &
\zcoh(Y_1) \ar[r] \ar[d]^{f_1} &
\zcoh(Y_0) \ar[r] \ar[d]^{f_0} &
0 \\
\cdots \ar[r] &
\zcoh(X_2) \ar[r] &
\zcoh(X_1) \ar[r] &
\zcoh(X_0) \ar[r] &
0
}
\]
where $f_{\bullet}$ is a homotopy equivalence.
Since the functor $\zcoh$ is full, there exists a corresponding diagram
\[
\xymatrix{
\cdots \ar[r] &
\mathbb{Z}[Y_2] \ar[r] \ar[d]^{f_2} &
\mathbb{Z}[Y_1] \ar[r] \ar[d]^{f_1} &
\mathbb{Z}[Y_0] \ar[r] \ar[d]^{f_0} &
0 \\
\cdots \ar[r] &
\mathbb{Z}[X_2] \ar[r] &
\mathbb{Z}[X_1] \ar[r] &
\mathbb{Z}[X_0] \ar[r] &
0
}
\]
along with the associated morphisms establishing that it is a homotopy
equivalence.
(However, this new diagram is only a homotopy equivalence of chain
complexes in the \emph{stable} category.)

Nevertheless, we can take duals and apply the Brauer group
Mackey functor to obtain
\[
\xymatrix{
0 \ar[r] &
\Br(X_0) \ar[d]^{f_0^\ast} \ar[r] &
\Br(X_1) \ar[d]^{f_1^\ast} \ar[r] &
\Br(X_2) \ar[d]^{f_2^\ast} \ar[r] &
\cdots \\
0 \ar[r] &
\Br(Y_0) \ar[r] &
\Br(Y_2) \ar[r] &
\Br(Y_1) \ar[r] &
\cdots
}
\]
which commutes and is a homotopy equivalence by
the stable Yoshida's Theorem~\ref{thm:stable_Yoshida}.
Therefore $f_0^\ast$ induces an isomorphism of the kernels
so the exactness of \eqref{eq:reminder} implies
the exactness of \eqref{eq:end_res_rel}.
\end{proof}

As we'll see below, one can construct exact sequences as above for
\emph{any} finitely generated cohomological/Tate Mackey functor.
The proof of the theorem above gives equivalences between the
corresponding complexes of relative/absolute Brauer groups.
However, even in the case where $M$ is the character lattice of torus
$T$ and we are resolving $\mathcal{H}^1(M)$, there is no connection 
to $H^1(k,T)$ unless the torus is retract rational.

\section{Permutation resolutions of Tate Mackey functors}
\label{sec:perm_resolutions}

The authors' main motivation for this paper was to
explicitly construct descriptions of $H^1(k,T)$ using Brauer groups.
Theorem~\ref{thm:main} allows us to reduce this question to
the study of Mackey functors over $\mathbb{Z}$.

From~Theorem~\ref{thm:comack_projective},
the category $\operatorname{CoMack}_R(G)$
has enough projectives, so we can construct projective resolutions.
Moreover, projective objects are always direct summands of objects of
the form $\mathcal{H}^0(X)$ where $X$ is a $G$-set.
Consequently, this is also true for $\operatorname{TMack}_R(G)$ using the
quotient objects $\widehat{\mathcal{H}}^0(X)$.

However, while Mackey functors over fields and complete discrete
valuation rings have a nice theory, the integral case
is less tractable.
In particular, we do not have a Krull-Schmidt theorem.
The category of Tate Mackey functors allows us considerably more
control.

\begin{prop} \label{prop:Tate_p_decomposition}
Let $G$ be a finite group of order $n$
and suppose $p_1,\ldots,p_r$ are the distinct primes dividing $n$.
We have canonical equivalences of categories
\[
\operatorname{TMack}_{\mathbb{Z}}(G) \cong
\operatorname{TMack}_{\mathbb{Z}/n\mathbb{Z}}(G)
\]
and
\[
\operatorname{TMack}_{\mathbb{Z}}(G) \cong
\operatorname{TMack}_{\mathbb{Z}_{p_1}}(G) \times
\cdots \times
\operatorname{TMack}_{\mathbb{Z}_{p_r}}(G).
\]
(Here $\mathbb{Z}_{p_i}$ denotes the $p_i$-adic integers.)
\end{prop}

\begin{proof}

Suppose $\mathcal{M}$ is in $\operatorname{TMack}_{\mathbb{Z}}(G)$.
Since $\mathcal{M}$ is cohomological, the map
\[
I^H_1 R^H_1 : \mathcal{M}(H) \to \mathcal{M}(H)
\]
is multiplication by $|H|$.  However, since $\mathcal{M}$ is Tate,
this map is also the zero map.
Thus, multiplication by $n$ is trivial for every subgroup $H$.

In view of Theorem~\ref{thm:stable_Yoshida},
the first equivalence now follows from the fact that 
the ring $\tau_{\mathbb{Z}}(G)$ has a $\mathbb{Z}/n\mathbb{Z}$-algebra
structure.
By the Chinese remainder theorem,
\[
\mathbb{Z}/n\mathbb{Z} \cong
\mathbb{Z}/p_1^{s_1}\mathbb{Z} \times \cdots
\times \mathbb{Z}/p_r^{s_r}\mathbb{Z},
\]
which then implies the second equivalence.
\end{proof}

We will use the shorthand notations
\begin{align*}
\operatorname{TMack}(G) &:= \operatorname{TMack}_{\mathbb{Z}}(G)\\
\operatorname{TMack}_p(G) &:= \operatorname{TMack}_{\mathbb{Z}_p}(G)\\
\tau(G) &:= \tau_{\mathbb{Z}}(G) \\
\tau_p(G) &:= \tau_{\mathbb{Z}_p}(G)
\end{align*}
for the remainder of the paper.

The ring $\tau(G)$ is finitely generated as a $\mathbb{Z}/|G|\mathbb{Z}$-module,
therefore we have the following:

\begin{prop}
The ring $\tau(G)$ is finite.
\end{prop}

In particular, $\tau(G)$ is an artinian ring and we obtain the following:

\begin{prop} \label{prop:Tate_Krull_Schmidt}
The category of finitely generated Tate Mackey functors over
$\mathbb{Z}$ or $\mathbb{Z}_p$ is Krull-Schmidt and
there are finitely many indecomposable projective objects.
\end{prop}

In particular, we have the existence of projective covers
and, therefore, minimal projective resolutions.
In fact, the indecomposable projective objects can be completely
classified.  This follows from the classification for the indecomposable
projective objects for cohomological Mackey functors over
$\mathbb{Z}_p$, which we describe in the next section.

\subsection{Trivial source modules}
\label{ssec:tsm}

We recall some basic facts about trivial source modules from
\cite{Linckelmann_2018} and \cite{Lassueur2023}.

\begin{definition}
Let $G$ be a group, $p$ be a prime, and $R$ be the ring
$\mathbb{F}_p$ or $\mathbb{Z}_p$.
A \emph{trivial source module} is an indecomposable summand
of a permutation $RG$-lattice.
\end{definition}

The possible trivial source modules can be constructed as follows.
Let $H$ be a $p$-subgroup of $G$, let $N_G(H)$ be the normalizer of
$H$ in $G$, and let $\overline{N}_G(H) := N_G(H)/H$.
Now consider a simple $\mathbb{F}_p\overline{N}_G(H)$-module $V$;
in other words, $V$ is an irreducible representation
of $\overline{N}_G(H)$ over $\mathbb{F}_p$.
Let $U$ be the projective cover of $V$ as a
$\mathbb{F}_p\overline{N}_G(H)$-module.
We may view $U$ as an $\mathbb{F}_pN_G(H)$-module by inflation.
We obtain a $\mathbb{F}_pG$-module $W$ as the Green correspondent of
$U$.
The module $W$ is a trivial source $\mathbb{F}_pG$-module.
Trivial source modules are liftable to $\mathbb{Z}_pG$-lattices
and we obtain a $\mathbb{Z}_pG$-lattice $T_{H,V}$.
In summary:

\begin{prop}
Up to isomorphism, the trivial source modules over $\mathbb{Z}_pG$ are
the modules $T_{H,V}$ where $H$ varies over all $p$-subgroups of $G$,
and $V$ varies over all simple $\mathbb{F}_p\overline{N}_G(H)$-modules.
\end{prop}

Note that each $T_{H,V}$ has vertex $H$.
In particular, this means that $T_{H,V}$ is a direct summand of the
permutation lattice $\mathbb{Z}_p[G/H]$.
Of course, $T_{H,V}$ is usually also a direct summand of several other
permutation lattices, which is what makes finding permutation
resolutions so interesting.

We can now state the classification of indecomposable projective Tate
Mackey functors:

\begin{prop}
The isomorphism classes of indecomposable projective objects of
$\operatorname{TMack}(G)$ are precisely the functors
\[
\mathcal{T}_{H,V} := \widehat{\mathcal{H}}^0(T_{H,V})
\]
where
\begin{enumerate}
\item $p$ varies over all primes dividing $|G|$,
\item $H$ varies over all non-trivial $p$-subgroups of $G$,
\item $V$ varies over all simple $\mathbb{F}_p N_G(H)/H$-modules, and
\item $T_{H,V}$ is the trivial source module over $\mathbb{Z}_p$
associated to $(H,V)$.
\end{enumerate}
\end{prop}

\begin{proof}
Simple Mackey functors were classified in \cite{Thvenaz1990SimpleMF};
the cohomological case can be found in \cite[Proposition 16.10]{T-W}.
For $\operatorname{CoMack}_{\mathbb{F}_p}(G)$, the simple functors
$\mathcal{S}_{H,V}$ are indexed by $p$-subgroups $H$ of $G$ along
with a simple $\mathbb{F}_p \overline{N}_G(H)$-module $V$.
The full construction is rather complex, but we only need that
$\mathcal{S}_{H,V}(H)=V$ and $\mathcal{S}_{H,V}(K)=0$ if $K < H$.

Their projective covers are given by
\[
\mathcal{H}^0(T_{H,V} \otimes \mathbb{F}_p) \to \mathcal{S}_{H,V}
\]
in the category of cohomological Mackey functors over $\mathbb{F}_p$.
Note that $\mathcal{S}_{H,V}$ is also a simple Mackey functor over $\mathbb{Z}_p$.
Since $T_{H,V} \to T_{H,V} \otimes \mathbb{F}_p$ is a projective cover
in the category of $\mathbb{Z}_pG$-modules,
the projective cover is 
\[
\mathcal{H}^0(T_{H,V}) \to \mathcal{S}_{H,V}
\]
in $\operatorname{CoMack}_{\mathbb{Z}_p}(G)$.

We now consider the quotients of these objects in
$\operatorname{TMack}_{\mathbb{Z}_p}(G)$ via the
functor $\pi$ from Proposition~\ref{prop:adjoint_CT}.
Since they are simple, either $\pi(\mathcal{S}_{H,V}) \cong \mathcal{S}_{H,V}$
or $\pi(\mathcal{S}_{H,V}) \cong 0$.
If $H\ne 1$, then $S_{H,V}(1)=0$ but $S_{H,V}(H)\ne 0$
so they are already non-zero Tate Mackey functors.
Thus, when $H \ne 1$, we obtain projective covers
\[
\widehat{\mathcal{H}}^0(T_{H,V}) \to \mathcal{S}_{H,V}
\]
in $\operatorname{TMack}_{\mathbb{Z}_p}(G)$.
However, when $H=1$, the module $T_{1,V}$ is a direct summand of
$\mathbb{Z}_p[G/1]=\mathbb{Z}_pG$ and therefore
$\widehat{\mathcal{H}}^0(T_{1,V})=0$.
Consequently, $\pi(\mathcal{S}_{1,V}) \cong 0$.

We have characterized the Tate Mackey functors over $\mathbb{Z}_p$
for all primes $p$.  The case for $\mathbb{Z}$
now follows from Proposition~\ref{prop:Tate_p_decomposition}.
\end{proof}

\subsection{Projective resolutions}
\label{ssc:proj_res}

Suppose $\mathcal{M}$ is a Tate Mackey Functor.  We have
a minimal projective resolution
\[
\cdots \to \mathcal{P}_2 \to \mathcal{P}_1 \to \mathcal{P}_0 \to M
\]
where
\[
\mathcal{P}_i \cong \bigoplus_{j=1}^r \left(\mathcal{T}_{H,V}\right)^{\oplus \beta_i(H,V)} .
\]
for uniquely determined non-negative integers $\beta_i(H,V)$.
If $N$ is the number of pairs $(H,V)$, then we may view each
$\beta_i$ as a vector in $\mathbb{Z}^N$ with non-negative entries.
Alternatively, we may view $\beta_i$ as a positive element in
$K_0(\tau(G))$.

The following is not original and is surely well known in much more
general contexts, but we could not find a direct reference:

\begin{prop} \label{prop:Betti}
Suppose $\mathcal{P}_\bullet \to \mathcal{M}$ is a minimal projective
resolution with invariants $\{ \beta_i \}$ in $K_0(\tau(G))$.
Suppose $\ldots, \mathcal{P}'_2, \mathcal{P}'_1, \mathcal{P}'_0$ are projective
Tate Mackey functors with corresponding invariants $\{ \beta'_i \}$.
There exists a projective resolution
\[
\cdots \to \mathcal{P}'_2 \to \mathcal{P}'_1 \to \mathcal{P}'_0 \to
\mathcal{M}  
\]
if and only if
\[
\sum_{i=0}^d (-1)^i \beta'_i(H,V) \ge \sum_{i=0}^d (-1)^i \beta_i(H,V)
\]
for all $(H,V)$ and all $d \ge 0$.
\end{prop}

\begin{proof}
Let $\Omega_i(M)$ be $i$th syzygy of the minimal
projective resolution.  In other words,
$\Omega_0(\mathcal{M})=\mathcal{M}$ and
we have exact sequences
\[
0 \to \Omega_{i+1}(\mathcal{M}) \to \mathcal{P}_i \to \Omega_i(\mathcal{M}) \to 0
\]
for every $i$.

Suppose the desired projective resolution exists where
$\Omega'_i(\mathcal{M})$ is the $i$th syzygy of
the other resolution.
Recall that we have projectives $\mathcal{Q}_i$ and $\mathcal{Q}'_i$ such that
\[
\Omega'_i(\mathcal{M}) \oplus \mathcal{Q}_i \cong \Omega_i(\mathcal{M}) \oplus
\mathcal{Q}'_i .
\]
Since the sequence for $\Omega_i(\mathcal{M})$ is minimal, it has no projective
summands.  Thus, by Krull-Schmidt, we have
\begin{equation} \label{eq:minimal_omega}
\Omega'_i(\mathcal{M}) \cong \Omega_i(\mathcal{M}) \oplus \mathcal{R}_i
\end{equation}
for some projective $\mathcal{R}_i$.

By Schanuel's Lemma, we have an isomorphism
\begin{equation} \label{eq:Schanuel}
\Omega'_{d+1}(\mathcal{M}) \oplus \mathcal{P}_d \oplus
\mathcal{P}'_{d-1} \oplus \mathcal{P}_{d-2} \cdots
\cong
\Omega_{d+1}(\mathcal{M}) \oplus \mathcal{P}'_d \oplus \mathcal{P}_{d-1}
\oplus \mathcal{P}'_{d-2} \cdots 
\end{equation}
Passing to K-theory, the two equations \eqref{eq:minimal_omega}
and \eqref{eq:Schanuel} imply
\[
[\mathcal{R}_d] + [\mathcal{P}_d] + [\mathcal{P}'_{d-1}] + [\mathcal{P}_{d-2}] \cdots =
[\mathcal{P}'_d] + [\mathcal{P}_{d-1}] + [\mathcal{P}'_{d-2}]
\]
after canceling $\Omega_{d+1}(\mathcal{M})$ from both sides.
Since $[\mathcal{R}_d]$ is positive,
we get the desired inequality.

For the converse, we simply build our resolution while establishing the
equations
\[
\Omega'_i \cong \Omega_i \oplus \mathcal{R}_i
\]
and
\[
\mathcal{P}'_i \cong \mathcal{P}_i \oplus \mathcal{R}_i \oplus
\mathcal{R}_{i+1}
\]
recursively as we go.
Indeed, comparing
\[
0 \to \Omega_{i+1}(\mathcal{M}) \to \mathcal{P}_i \to
\Omega_i(\mathcal{M}) \to 0
\]
and
\[
0 \to \Omega'_{i+1}(\mathcal{M}) \to \mathcal{P}'_i \to
\Omega_i(\mathcal{M}) \oplus \mathcal{R}_i \to 0
\]
the desired equations follow quickly from the fact that
$\mathcal{P}_i \to \Omega_i(\mathcal{M})$ is a projective cover and
$\mathcal{R}_i$ is projective.
\end{proof}

\subsection{Permutation resolutions}
\label{ssec:perm_res}

In view of Theorem~\ref{thm:main_detailed}, given an invertible
$G$-lattice $M$, we want to be able to compute permutation presentations
of $\mathcal{H}^1(M)$ in the following sense.

\begin{definition}
Let $\mathcal{M}$ be a Tate Mackey functor.
A \emph{permutation resolution} of $\mathcal{M}$ is an exact sequence of
Tate Mackey functors
\[
\cdots \to \widehat{\mathcal{H}}^0(X_2) \to \widehat{\mathcal{H}}^0(X_1)
 \to \widehat{\mathcal{H}}^0(X_0) \to \mathcal{M} \to 0
\]
where $X_0, X_1, X_2, \ldots$ are $G$-sets.
A \emph{permutation presentation} of $\mathcal{M}$ is a permutation
resolution with only two $G$-sets $X_0,X_1$.
\end{definition}

Since permutation Tate Mackey functors are in particular projective,
we can use Proposition~\ref{prop:Betti}.
In the remainder of this section, we sketch how to determine all
permutation resolutions for an arbitrary Tate Mackey functor.
An extended example is worked out in Section~\ref{sec:blunk_example}.

\subsubsection{Determine decompositions of permutation lattices}

We need to be able to determine the decomposition
\begin{equation} \label{eq:perm_decomp}
\widehat{\mathcal{H}}^0(\mathbb{Z}[G/K])
\cong \bigoplus_{H,V} \left(\mathcal{T}_{H,V}\right)^{\oplus m^K_{H,V}} 
\end{equation}
where $m^K_{H,V}$ are non-negative integers.
Equivalently, an explicit description of the map
\begin{equation} \label{eq:perm_decomp_Ktheory}
K_0(\permcat{\mathbb{Z}G}) \to K_0(\tau(G))
\end{equation}
as a matrix in the natural bases
has the integers $m^K_{H,V}$ as its entries.

This can be done directly using the theory of Mackey functors
(for example, using \cite[8.6, 8.8]{T-W}).
However, this is essentially about the representation
theory of $\mathbb{F}_pG$-modules.
Indeed, for each prime $p$ dividing $|G|$, we compute the decomposition
of $\mathbb{F}_p[G/K]$ into indecomposable summands, which must be trivial
source modules, then lift to $\mathbb{Z}_p$.
We find that
\[
\mathbb{F}_p[G/K]
\cong \bigoplus_{H,V} \left(T_{H,V} \otimes \mathbb{F}_p \right)^{\oplus
n^K_{H,V}} 
\]
where the $n^K_{H,V}$ agree with the multiplicities $m^K_{H,V}$ from
\eqref{eq:perm_decomp}, except that $n^K_{H,V}$ only makes sense for $H$
a $p$-group and $m^K_{H,V}$ only makes sense when $H \ne 1$.

These decompositions can be computed explicitly using 
the \texttt{meataxe} in GAP \cite{GAP4}.
Alternatively, one can determine multiplicities by computing the image
of permutation lattices in the trivial source ring and then decomposing
them using the trivial source character tables
(see \cite[\S{5.5}]{BensonI} and \cite{Boehmler2024}).

\subsubsection{Compute a minimal projective resolution}

Since the finitely generated subcategory of $\operatorname{TMack}(G)$
is a Krull-Schmidt category, minimal projective resolutions exist.
Computing them amounts to iteratively finding a projective cover of a Tate Mackey
functor.
In order to compute these in practice, it suffices to consider
each prime $p$ separately.
One then considers successive Brauer quotients of $p$-groups, starting
with the Sylow $p$-subgroup and then working down the subgroup lattice.
Since the Green correspondent of each trivial source module with vertex
$H$ is a projective module in $\mathbb{Z}_pH$, we just compute
projective covers for $\mathbb{F}_pH$-modules.

\subsubsection{Determine permutation resolutions}

Suppose $\beta_0, \beta_1,\ldots$ are the Betti vectors in
$K_0(\tau(G))$ of the minimal
projective resolution of $\mathcal{M}$.
From \eqref{eq:perm_decomp_Ktheory}, we can compute the multiplicities
for $\widehat{\mathcal{H}}^0(X)$ for any $G$-set $X$.
Thus, in view of Proposition~\ref{prop:Betti},
we can compute all possible permutation resolutions
by simply comparing vectors for $X_0,X_1,\ldots$ and
$\beta_0,\beta_1,\ldots$ and recursively determining all the
possibilities.

\section{Revisiting del Pezzo surfaces of degree 6}
\label{sec:blunk_example}

In this section, we demonstrate how the techniques discussed above can
be used to describe tori using Brauer groups.
Specifically, we show how Blunk's classification of del Pezzo
surfaces of degree 6 could be rediscovered using this technique.

\subsection{Trivial source modules for the dihedral group of order 12}
\label{ssec:tsm_d12}

Let $G$ be a finite group isomorphic to $D_{12} \cong S_3 \times S_2$.
We pick generators $\sigma, \tau, \delta$
where
\[
\tau \mapsto (12),\quad
\delta \mapsto (123),\quad
\sigma \mapsto (45)
\]
exhibits an embedding of $G$ into $S_5$.
We list and name all conjugacy classes of nontrivial proper subgroups of
$G$ in Table~\ref{tab:D6conj}.

\begin{table}[ht]
\centering
\begin{tabular}{|l|l|l|}
\hline
Structure & Name & Representative\\
\hline \hline
$C_2$ & \legend{A1} & $\langle \sigma \rangle$ \\
 & \legend{A2} & $\langle \tau \rangle$ \\
 & \legend{A3} & $\langle \sigma\tau \rangle$ \\
\hline
$C_3$ & \legend{B} &$ \langle \delta \rangle$ \\
\hline
$C_2 \times C_2$ & \legend{C} & $\langle \sigma, \tau \rangle$ \\
\hline
$C_6$ & \legend{D} & $\langle \sigma, \delta \rangle$ \\
\hline
$S_3$ & \legend{E1} & $\langle \tau, \delta \rangle$ \\
 & \legend{E2} & $\langle \sigma\tau, \delta \rangle$ \\
\hline
\end{tabular}
\caption{Conjugacy classes of subgroups in $D_6$}
\label{tab:D6conj}
\end{table}

Next, we enumerate all trivial source modules with non-trivial vertex
for $G$ over $\mathbb{Z}_p$ where $p$ is either $2$ or $3$.
Recall that each trivial source module $T_{H,V}$ corresponds to a pair
$(H,V)$ where $H$ is a $p$-subgroup and $V$ is a simple
$\mathbb{F}_pN_G(H)/H$-module.
We list all the possibilities in Table~\ref{tab:ts_def}, where we use
an ad-hoc labeling for the simple modules.

\newcommand{\iden}{$\begin{pmatrix} 1 & 0\\ 0 & 1 \end{pmatrix}$}
\newcommand{\swap}{$\begin{pmatrix} 0 & 1\\ 1 & 0 \end{pmatrix}$}
\newcommand{\rota}{$\begin{pmatrix} 0 & -1\\ 1 & -1 \end{pmatrix}$}
\newcommand{\nulo}{$\begin{pmatrix} 1 \end{pmatrix}$}
\newcommand{\nego}{$\begin{pmatrix} -1 \end{pmatrix}$}

\begin{table}[ht]
\centering
\begin{tabular}{|l|c||c|c|c|}
\hline
Label& $T_{H,V}$ & $\sigma$ & $\tau$ & $\delta$ \\
\hline
${\legend{A1},1}$ & $\mathbb{Z}_2^2$ &
\iden & \swap & \iden \\
${\legend{A1},2}$ & $\mathbb{Z}_2^2$ &
\iden & \swap & \rota \\
${\legend{A2},1}$ & $\mathbb{Z}_2^2$ &
\swap & \iden & \iden \\
${\legend{A3},1}$ & $\mathbb{Z}_2^2$ &
\swap & \swap & \iden \\
${\legend{C},1}$ & $\mathbb{Z}_2$ &
\nulo & \nulo & \nulo \\
\hline
${\legend{B},0}$ & $\mathbb{Z}_3$ &
\nulo & \nulo & \nulo \\
${\legend{B},1}$ & $\mathbb{Z}_3$ &
\nulo & \nego & \nulo \\
${\legend{B},2}$ & $\mathbb{Z}_3$ &
\nego & \nulo & \nulo \\
${\legend{B},3}$ & $\mathbb{Z}_3$ &
\nego & \nego & \nulo \\
\hline
\end{tabular}
\caption{Trivial source modules with non-trivial vertex}
\label{tab:ts_def}
\end{table}

For each transitive $G$-set $X=G/K$, we have a decomposition
\[
\widehat{\mathcal{H}}^0(X) = \bigoplus_{H,V}
\left(\mathcal{T}_{H,V} \right)^{m_{H,V}} .
\]
This can be determined by decomposition the permutation
lattice $\mathbb{Z}_p[G/K]$
as a direct sum of trivial source modules
and removing those with trivial vertex.
The multiplicities can be found in Table~\ref{tab:perm_to_ts}.

\begin{table}[ht]
\centering
\begin{tabular}{|c| c c |c |c | c | c c c c|}
\hline
{\footnotesize{ Vertex $H$ } }& \multicolumn{2}{c}{\legend{A1}} & \legend{A2} & \legend{A3} & \legend{C}
& \multicolumn{4}{c|}{\legend{B}} \\
\hline
{\footnotesize{Simple $V$} } & 1 & 2 & 1 & 1 & 1 & 0 & 1 & 2 & 3 \\
\hline
$G/\legend{A1} $ & 1  & 2  &  &   &   &   &   &   &  \\
$G/\legend{A2} $  &   &  & 1 &       &     &      &    &     &   \\
$G/\legend{A3} $  &  &   &   &  1   &     &     &      &     &    \\
$G/\legend{C} $  &   & 1  &  &   & 1   &   &    &    &    \\
$G/\legend{B} $  &  &  &  &    &    &   1    &   1  &   1  &  1   \\
$G/\legend{D}  $  & 1  &   &  &     &    &  1   &  1   &   &    \\
$G/\legend{E1} $  &   &  & 1 &     &   & 1   &    &   1   &    \\
$G/\legend{E2} $  &  &  &  &    1  &     & 1 &    &    &   1 \\
$G/G $  &  &   &  &      &  1  &  1   &      &     &     \\
\hline
\end{tabular}
\caption{Decomposition of permutation Tate Mackey functors in terms of
indecomposable projectives}
\label{tab:perm_to_ts}
\end{table}

\subsection{Resolving the character lattice}
\label{ssec:char}

Consider the $G$-lattice $M$ which has underlying abelian group
$\mathbb{Z}^2$ and whose $G$-action is given by
\[
\tau \mapsto \begin{pmatrix} 0 & 1\\ 1 & 0 \end{pmatrix} ,\quad
\delta \mapsto \begin{pmatrix} 0 & -1\\ 1 & -1 \end{pmatrix} ,\quad
\sigma \mapsto \begin{pmatrix} -1 & 0\\ 0 & -1 \end{pmatrix}.
\]
This is the character lattice of the torus $T$ in a del Pezzo surface of
degree $6$ where $G$ is the group of fan automorphisms.
By comparing group cohomology of each side, one checks that
\[
\mathcal{H}^1(M) \cong
\mathcal{T}_{A1,2}
\oplus \mathcal{T}_{B,2}.
\]
Thus, the minimal projective resolution has length $1$ in this case.

We want to build a minimal permutation resolution 
\[
\ldots \widehat{\mathcal{H}}^0(X_1) \to \widehat{\mathcal{H}}^0(X_0) \to
\widehat{\mathcal{H}}^1(M)
\]
for $G$-sets $X_0$ and $X_1$.

Consulting Table~\ref{tab:perm_to_ts}, the minimal choices for $X_0$ are
\[
G/\legend{A1} \sqcup G/\legend{B},
\quad
G/\legend{C} \sqcup G/\legend{B},
\quad
G/\legend{A1} \sqcup G/\legend{E1},
\quad
G/\legend{C} \sqcup G/\legend{E1} .
\]
All of them lead to different permutation resolutions,
but we will choose the last one (which will lead to Blunk's theorem).
Observe that
\[
\widehat{\mathcal{H}}^0(G/\legend{C} \sqcup G/\legend{E1})
\cong \left(\mathcal{T}_{\legend{A1},2}
\oplus \mathcal{T}_{\legend{C},1}\right)
\oplus \left( \mathcal{T}_{\legend{A2},1}
\oplus \mathcal{T}_{\legend{B},0}
\oplus \mathcal{T}_{\legend{B},2}\right).
\]
The kernel of the map onto $\widehat{\mathcal{H}}^0(M)$ is isomorphic to:
\[
\mathcal{T}_{\legend{C},1}
\oplus \mathcal{T}_{\legend{A2},1}
\oplus \mathcal{T}_{\legend{B},0}.
\]
This is isomorphic to
$\widehat{\mathcal{H}}^0(G/\legend{A2} \sqcup G/G)$.
Therefore, we have an exact sequence
\[
0 \to \widehat{\mathcal{H}}^0(G/\legend{A2} \sqcup G/G)
\to \widehat{\mathcal{H}}^0(G/\legend{C} \sqcup G/\legend{E1})
\to \widehat{\mathcal{H}}^0(M) \to 0,
\]
which gives a permutation presentation.

Let us now connect this to the Brauer group.
Let $K/k$ be a field extension with Galois group $D_{12}$,
define
\[
F_2 := K^{\langle \tau,\delta \rangle}=K^{\legend{E1}},\quad
F_3 := K^{\langle \tau, \sigma \rangle}=K^{\legend{}C},
\]
and observe that
\[
k = K^G,\quad
F_2F_3 = K^{\langle \tau \rangle} = K^{\legend{A2}}.
\]
Here we abuse notation by conflating the conjugacy classes
$\legend{A}$, $\legend{C}$ and $\legend{E1}$
with their representative in the table.
Appealing to Theorem~\ref{thm:main_existence_relative},
we conclude that
\[
H^1(k,T) \cong \operatorname{ker} \left( \Psi:
\rBr(F_2) \times \rBr(F_3)
\to \rBr(F_2F_3) \times \rBr(k)
\right).
\]

By finding explicit decompositions of the permutation modules into
trivial source modules and applying Yoshida's theorem,
we get an explicit expression
\[
\Psi(B,Q) =
\left(
\operatorname{Res}_{F_2F_3/F_2}(B),\
\operatorname{Cor}_{F_2/k}(B) + \operatorname{Cor}_{F_3/k}(Q) \right)
.
\]

This description uses the relative Brauer group, but we easily get 
a description for the absolute Brauer group:

\begin{prop} \label{prop:ourBlunk}
Let $T$ be a $k$-torus with a faithful open orbit on a del Pezzo surface of
degree $6$.  Suppose $F_3$ is the cubic \'etale $k$-algebra and
$F_2$ is the quadratic \'etale $k$-algebra defining the isomorphism
class of $T$.
There is a bijection between elements of $H^1(k,T)$ and pairs $(B,Q)$
where $B \in \Br(F_2)$, $Q \in \Br(F_3)$ such that the following hold:
\begin{enumerate}
\item $\operatorname{Res}_{K/F_2}(B)=0$,
\item $\operatorname{Res}_{K/F_3}(Q)=0$,
\item $\operatorname{Res}_{F_2F_3/F_2}(B)=0$, and
\item $\operatorname{Cor}_{F_2/k}(B) + \operatorname{Cor}_{F_3/k}(Q)=0$.
\end{enumerate}
\end{prop}

Other choices of permutation presentations will give different
descriptions.

\subsection{Comparison to Blunk's classification}
\label{ssec:blunk_compare}

Blunk's result states that $(B,Q)$ satisfy the relations
\begin{enumerate}
\item $\operatorname{Res}_{F_2F_3/F_2}(B)=0$,
\item $\operatorname{Res}_{F_2F_3/F_3}(Q)=0$,
\item $\operatorname{Cor}_{F_2/k}(B)=0$, and
\item $\operatorname{Cor}_{F_3/k}(Q)=0$.
\end{enumerate}
At first, these seem like stronger conditions than
Proposition~\ref{prop:ourBlunk} above,
but they are the same.

Indeed, $\operatorname{Res}_{K/F_2}(B)$ splitting implies that $B$ is
$3$-torsion, while $\operatorname{Res}_{K/F_3}(Q)$ splitting implies
that $Q$ is $4$-torsion.
Therefore,
$\operatorname{Cor}_{F_2/k}(B)$ and $\operatorname{Cor}_{F_3/k}(Q)$
split independently as soon as their sum splits.
Moreover, by the Mackey decomposition, we have
\[
\operatorname{Res}_{F_2F_3/k}\left(
\operatorname{Cor}_{F_3/k}(Q)\right)=
\operatorname{Res}_{F_2F_3/F_3}(Q) +
\operatorname{Cor}_{K/F_2F_3}\left(
\delta^\ast \left(\operatorname{Res}_{K/F_3}(Q)\right)\right).
\]
Since $\operatorname{Res}_{K/F_3}(Q)$ and $\operatorname{Cor}_{F_3/k}(Q)$
split, we conclude that $\operatorname{Res}_{F_2F_3/F_3}(Q)$ splits.


\bibliographystyle{alpha}

\end{document}